\documentclass[12pt,a4paper]{article}
\usepackage[centertags]{amsmath}
\usepackage{amsfonts}
\usepackage{amssymb}
\usepackage{amsthm}
\usepackage{color}
\usepackage{marginnote}
\usepackage{float}

\usepackage{bussproofs}

\usepackage{fullpage}
\usepackage{tikz}
\usepackage{verbatim}
\usetikzlibrary{arrows}

\newtheorem{defi}{Definition}

\newtheorem{rem}{Remark}
\newtheorem{prop}{Proposition}
\newtheorem{lem}{Lemma}

\usepackage{color}

\begin{document}

\title{On distributive join-semilattices}

\author{Rodolfo C. {E}rtola-Biraben$^1$, Francesc Esteva$^2$, and Llu\'is Godo$^2$ \vspace{0.7cm} \\ 
$^1$ CLE - State University of Campinas \\ 
13083-859 Campinas, S\~ao Paulo, Brazil   \vspace{0.3cm} \\
$^2$  IIIA - CSIC, 08193 Bellaterra, Spain
}

 \date{}

\maketitle

\begin{abstract}
 Motivated by Gentzen's disjunction elimination rule in his Natural Deduction calculus and 
 reading inequalities with meet in a natural way, 
 we conceive a notion of distributivity for join-semilattices. 
 We prove that it is equivalent to a notion present in the literature. 
 In the way, we prove that those notions are linearly ordered. 
 We finally consider the notion of distributivity in join-semilattices with arrow, that is, 
 the algebraic structure corresponding to the disjunction-conditional fragment of intuitionistic logic.  
\end{abstract}

\section{Introduction}

Different notions of distributivity for semilattices have been proposed in the literature 
as a generalization of the usual distributive property in lattices. 
As far as we know, notions of distributivity for semilattices have been given, in chronological order, 
by Gr\"atzer and Schmidt \cite{GS} in 1962, by Katri\v{n}\'ak \cite{K} in 1968, by Balbes \cite{B} in 1969, by Schein\cite{S} in 1972, 
by Hickman \cite{H} in 1984, and by Larmerov\'a and Rach\r{u}nek  \cite{LR} in 1988. 
Following the names of its authors, we will use the terminology GS-, K-, B-, S$_n$-, H-, and LR-distributivity, respectively. 


In this paper, motivated by Gentzen's disjunction elimination rule in his Natural Deduction calculus, 
and reading inequalities with meet in a natural way, 
we conceive another notion of distributivity for join-semilattices, that we call ND-distributivity. 
We aim to find out whether it is equivalent to any of the notions already present in the literature.
In doing so, we also compare the different notions of distributivity for join-semilattices we have found. 
Namely, we see that the given notions imply each other in the following linear order:


\smallskip

\begin{center}
GS $\Rightarrow$ K $\Rightarrow$ (H $\Leftrightarrow$ LR $\Leftrightarrow$ ND) $\Rightarrow$ B $\Rightarrow \cdots$ S$_n 
\Rightarrow$ S$_{n-1} \Rightarrow \cdots$ S$_3 \Rightarrow$ S$_2$, 
\end{center}

\smallskip

\noindent and we also provide countermodels for the reciprocals. 

Additionally, 
we show that H-distributivity may be seen as a very natural translation of a way to define distributivity for lattices, 
fact that will provide more motivation for the use of that notion. 
Note that Hickman used the term mild distributivity for H-distributivity. 
 

The paper is structured as follows. 
After this introduction, in Section 2 we provide some notions and notations that will be used in the paper. 
In Section 3 we show how to arrive to our notion of ND-distributivity for join-semilattices. 
In Section 4 we compare the different notions of distributivity for join-semilattices that appear in the literature. 
We prove that one of those is equivalent to the notion of ND-distributivity found in Section 3. 
Finally, in Section 5 we consider what happens with the different notions of distributivity considered in Section 4 
when join-semilattices are  expanded with a natural version of the relative meet-complement. 

\section{Preliminaries}

In this section we provide the basic notions and notations that will be used in the paper. 

Let ${\bf{J}} = (J; \leq)$ be a poset. 
For any $S \subseteq J$, we will use the notations $S^l$ and $S^u$ to denote the set of lower and upper bounds of $S$, respectively. 
That is, 

\begin{itemize}
\item[] $S^l = \{x \in J: x \leq s$, for all $s \in S \}$ and 

$S^u = \{x \in J: s \leq x$, for all $s \in S \}$.   
\end{itemize}

\begin{lem} \label{BL}
Let ${\bf{J}} = (J; \leq)$ be a poset. 
For all $a, b, c \in J$ the following statements are equivalent: 
\begin{itemize}
\item[(i)] for all $x \in J$, if $x \leq a$ and $x \leq b$, then $x \leq c$, 
 
\item[(ii)] $\{a, b\}^l \subseteq \{c\}^l$,
 
\item[(iii)]  $c \in \{a, b\}^{lu}$.
\end{itemize}
\end{lem}

A poset ${\bf{J}} = (J; \leq)$ is a {\em join-semilattice} (resp. meet-semilattice) 
if $\sup\{a, b\}$ (resp. $\inf\{a, b\}$) exists for every $a, b \in J$. 
A poset ${\bf{J}} = (J; \leq)$ is a lattice if it is both a join- and a meet-semilattice. 
As usual, the notations $a \vee b$ (resp.  $a \wedge b$) shall stand for  $\sup\{ a, b\}$ (resp. $\inf\{a, b \}$).  

%
%
%
%
%

%

Given a join-semilattice ${\bf{J}} = (J; \leq)$, we will use the following notions:

\begin{itemize}

\item  ${\bf{J}}$ is \emph{downwards directed} iff 
for any $a, b \in J$, there exists $c \in J$ such that $c \leq a$ and $c \leq b$. 

\item A non empty subset $I \subseteq J$ is said to be an {\em ideal} iff \\
(1) if $x,y \in I$, then $x\vee y \in I$ and \\
(2) If $x \in I$ and $y \leq x$, then $y \in I$.
\item The principal ideal generated by an element $a \in A$, noted $(a]$, is defined
by $(a] = \{x \in A : x \leq a\}$.
\item $Id({\bf{J}})$ will denote the set of all ideals of ${\bf{J}}$.
\item  $Id_{fp}({\bf{J}})$ will denote the subset of ideals that are intersection of a finite set of principal ideals, that is, 
$Id_{fp}({\bf{J}}) = \{(a_1] \cap \dots \cap (a_k] : a_1,...a_k \in J\}$.
\end{itemize}

In this paper we are concerned with various notions of distributivity for join-semilattices, 
all of them generalizing the usual notion of distributive lattice, that is, 
a lattice ${\bf{J}} = (J; \leq)$ is distributive if the following equation holds true for any elements $a, b, c \in J$: 

\begin{itemize}
\item[{\bf{(D)}}] $a \land (b \lor c) = (a \land b) \lor (a \land c)$ (equivalently, $a \lor (b \land c) = (a \lor b) \land (a \lor c)$). 
\end{itemize}

\noindent There are several equivalent formulations of this property, 
in particular we mention the following ones that are relevant for this paper:

\begin{itemize}
\item for all $a, b, c \in J$, if $a \lor b = a \lor c$ and $a \land b = a \land c$ then $ b = c$. 

\item for any two ideals $I_1, I_2$ of $\bf J$, the ideal $I_1 \lor I_2$ generated by their union is defined by 
$I_1 \lor I_2 = \{ a \lor b : a \in I_1, b \in I_2\}.$

\item the set  $Id({\bf{J}})$ of ideals of $\bf J$ is a distributive lattice.
\end{itemize} 

\noindent For the case of semilattices, several non-equivalent generalizations of these conditions can be found in the literature, 
already mentioned in the introduction. 
However, as expected, all of them turn to be equivalent to usual distributivity in the case of lattices.

The class of distributive lattices form a variety (that is, an equational class). 
In contrast, in any sense of distributivity for join-semilattices that coincides with usual distributivity in the case of a lattice, 
the class of distributive join-semilattices is not even a quasi-variety. 
Indeed, consider the distributive lattice in Figure \ref{Fdjns}. 
Taken as a join-semilattice, the set of black-filled nodes is a sub join-semilattice, that is clearly a non-distributive lattice (a diamond). 
Thus, it is neither distributive as a join-semilattice. 
This proves that the class of distributive (in any sense that coincides with usual distributivity in the case of a lattice) 
join-semilattices is not closed by subalgebras, and hence it is not a quasi-variety.

\begin{figure} [ht]
\begin{center}

\begin{tikzpicture}

    \tikzstyle{every node}=[draw, circle, fill=white, minimum size=4pt, inner sep=0pt, label distance=0.5mm]


    \draw (0,0) node (0) [fill=black] {};
    \draw (-1.3,1.3) node (1) {};
    \draw (0,1.3) node (2) {};
    \draw (1.3,1.3) node (3) {};
    \draw (-1.3,2.6) node (4) [fill=black] [label=left:a] {};
    \draw (0,2.6) node (5) [fill=black] [label=right:b] {};
    \draw (1.3,2.6) node (6) [fill=black] [label=right:c] {};
    \draw (0,3.9) node (7) [fill=black] {};
    \draw (0)--(1)--(4)--(2)--(0)--(3)--(6)--(2)
    	  (1)--(5)--(3) (4)--(7)--(6)  (5)--(7);
 
\end{tikzpicture}

\end{center}
\caption{\label{Fdjns} A distributive lattice with a non-distributive sub join-semilattice. 
}
\end{figure}
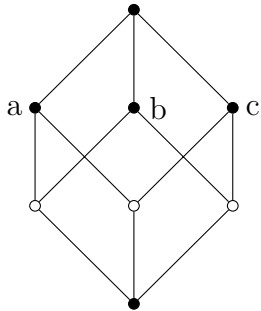


\section{Distributivity and Natural Deduction} \label{DaND}

{
Let us consider the disjunction-fragment of intuitionistic logic 
in the context of Gentzen's Natural Deduction calculus (see \cite[p. 186]{Ge}). 
It has the following introduction rule for $\vee$ and an analogous one with $\mathfrak{B}$ as only premiss: 

\begin{prooftree}
 \AxiomC{$\mathfrak{A}$}
 \LeftLabel{{\bf($\vee$I):}}
 \UnaryInfC{$\mathfrak{A \lor B}$}
\end{prooftree}

\noindent and the following disjunction elimination rule: 

\begin{prooftree}
 \AxiomC{$\mathfrak{A \lor B}$}
 \AxiomC{[$\mathfrak A$]}
 \noLine
 \UnaryInfC{$\mathfrak C$}
 \AxiomC{[$\mathfrak B$]}
 \noLine
 \UnaryInfC{$\mathfrak C$}
 \TrinaryInfC{$\mathfrak C$}
\end{prooftree}

\noindent The last rule may be read as saying that 
if $\mathfrak C$ follows from $\mathfrak A$ and $\mathfrak C$ follows from $\mathfrak B$, 
then $\mathfrak C$ follows from $\mathfrak A \vee B$, so reflecting what is usually called ``proof by cases''. 
It is possible to give an algebraic translation in the context of a join-semilattice ${\bf{J}} = (J; \leq)$: 

\smallskip

for all $a, b, c \in J$, if $a \leq c$ and $b \leq c$, then $a \vee b \leq c$,   

\smallskip

\noindent which is easily seen to be one of the conditions stating that $\vee$ is the supremum of $a$ and $b$.  
Now, the last rule is usually employed in a context with a fourth formula $\mathfrak H$: 

\begin{prooftree}
 \AxiomC{$\mathfrak H$, $\mathfrak{A \vee B}$}
 \AxiomC{$\mathfrak H$, [$\mathfrak A$]}
 \noLine
 \UnaryInfC{$\mathfrak C$}
 \AxiomC{$\mathfrak H$, [$\mathfrak B$]}
 \noLine
 \UnaryInfC{$\mathfrak C$}
 \LeftLabel{{\bf($\vee$E):}}
 \TrinaryInfC{$\mathfrak C$}
\end{prooftree}

\noindent In the context of a lattice ${\bf{L}} = (L; \leq)$, we would give the following algebraic translation: 

\smallskip

\begin{itemize}
\item[{\bf{(D$_{\wedge \vee}$)}}]  for all $h, a, b, c \in L$, 

if $h \wedge a \leq c$ and $h \wedge b \leq c$, then $h \wedge (a \vee b) \leq c$. 
\end{itemize}
\smallskip 

\noindent It is easily seen that {\bf(D$_{\wedge \vee}$)} is equivalent to the usual notion of distributivity for lattices. 
Now, the natural question arises how to give an algebraic translation of {\bf($\vee$E)} if only $\vee$ is available, 
for example, if we are in the context of a join-semilattice. 

%
%
%
%
%

Considering that an inequality $u \land v \leq w$ in a lattice ${\bf{L}} = (L; \leq)$ is equivalently expressed as the first order statement  

\begin{center}
for all $x \in L$, if $x \leq u$ and $x \leq v$ then $x \leq w$,
\end{center}  
we may write  {\bf{(D$_{\wedge \vee}$)}}  in the context of a join-semilattice ${\bf{J}} = (J; \leq)$ as follows: 

\smallskip

\begin{itemize}
\item[{\bf{(D$_{\vee}$)}}] for all $h, a, b, c \in J$, 

IF for all $x \in J$ (if $x \leq h$ and $x \leq a$, then $x \leq c)$ and 

\hspace*{0.4cm} for all $x \in J$ (if $x \leq h$ and $x \leq b$, then $x \leq c)$, 

THEN for all $x \in J$ (if $x \leq h$ and $x \leq a \vee b$, then $x \leq c$). 
\end{itemize}

\smallskip

\noindent Alternatively, using the equivalence between parts (i) and (ii) in Lemma \ref{BL}, we may write 

\smallskip

\begin{itemize}
\item[{\bf{(D$_{\vee}$)}}] for all $h, a, b, c \in J$,

 if $\{h,a\}^l \cup \{h,b\}^l \subseteq \{c\}^l$, then $\{h,a \vee b\}^l \subseteq \{c\}^l$. 
\end{itemize}

\smallskip

\noindent Yet, using the equivalence between parts (ii) and (iii) in Lemma \ref{BL}, we may also alternatively write 

\smallskip

\begin{itemize}
\item[{\bf{(D$_{\vee}$)}}] for all $h, a, b, c \in J$,

if $ c \in \{h,a\}^{lu} \cup \{h,b\}^{lu}$, then $c \in \{h,a \vee b\}^{lu}$. 
\end{itemize}

\smallskip

\noindent Accordingly, given the above logical motivation, it is natural to consider 
the following notion of distributivity for join-semilattices.

\begin{defi} 
A join-semilattice ${\bf{J}} = (J; \leq)$ is called ND-distributive (ND for Natural Deduction) if it satisfies  {\bf{(D$_{\vee}$)}}. 
\end{defi}

Now, it happens that there are many different (and non-equivalent) notions of distributivity for semilattices. 
This is not new:

\smallskip

\begin{quote}
``The concept of distributivity permits different non-equivalent generalizations from lattices to semilattices.'' 
(see \cite{S})
\end{quote}
\smallskip

\noindent So, it is natural to inquire whether the given notion of ND-distributivity for join-semilattices 
is equivalent to any of the notions already present in the literature and, if so, to which.
In what follows we will solve that question. 
In doing so, we will also compare the different notions of distributivity for join-semilattices that we have found.

In this paper, given our logical motivation, we restrict ourselves to study the distributivity property in join-semilattices, 
but an analogous path could be followed for meet-semilattices or even for posets.

\begin{rem} \em
Let us note that the following rule (reflecting proof by three cases) is equivalent to {\bf($\vee$E)}:

\begin{prooftree}
 \AxiomC{$\mathfrak H$, $\mathfrak{A \vee B \vee C}$}
 
 \AxiomC{$\mathfrak H$, [$\mathfrak A$]}
 \noLine
 \UnaryInfC{$\mathfrak D$}
 
 \AxiomC{$\mathfrak H$, [$\mathfrak B$]}
 \noLine
 \UnaryInfC{$\mathfrak D$}
 
 \AxiomC{$\mathfrak H$, [$\mathfrak C$]}
 \noLine
 \UnaryInfC{$\mathfrak D$}
 
 \QuaternaryInfC{$\mathfrak D$}
\end{prooftree}

\noindent Indeed, it implies {\bf($\vee$E)} taking $\mathfrak C = \mathfrak B$. 
Also, the following derivation shows that it may be derived using {\bf($\vee$E)} twice: 

\begin{prooftree}
 \AxiomC{$\mathfrak H$, $\mathfrak{A \vee B \vee C}$}
 
 \AxiomC{$\mathfrak H$, [$\mathfrak A$]}
 \noLine
 \UnaryInfC{$\mathfrak D$}
 
 \AxiomC{$\mathfrak H$, $\mathfrak{B \vee C}$}
 \AxiomC{$\mathfrak H$, [$\mathfrak B$]}
 \noLine
 \UnaryInfC{$\mathfrak D$}
 \AxiomC{$\mathfrak H$, [$\mathfrak C$]}
 \noLine
 \UnaryInfC{$\mathfrak D$}
 \TrinaryInfC{$\mathfrak D$}
 
 \TrinaryInfC{$\mathfrak D$}
\end{prooftree}
\end{rem}

%
%
}

\section{Different notions of distributivity for join-semilattices} \label{SDN}

In the following subsections we consider and compare the notions of distributivity for semilattices 
we have found in the literature. 
Some authors have presented their notion for the case of meet-semilattices and others for join-semilattices. 
We will make things uniform and, motivated by the logical considerations in the previous section, 
we will choose to consider join-semilattices.  

We emphasize that all the distributivity notions for semilattices (and posets) proposed in the literature 
are generalizations of the distributivity property for lattices, in fact, when restricted to lattices all these notions coincide.

\subsection{GS-distributivity}

The following seems to be the most popular definition of distributivity for join-semilattices. 

\begin{defi} \label{GSd}
A join-semilattice ${\bf{J}} = (J; \leq)$ is GS-distributive iff 

\begin{itemize}
\item[{\bf (GS)}]  for all $a, b, x \in J$, if $x \leq a \vee b$, then there exist $a', b' \in J$ such that 
$a' \leq a$, $b' \leq b$, and $x = a' \vee b'$. 
\end{itemize}
\end{defi}

\smallskip

\noindent In order to visualize it, see Figure \ref{Dmsl}. 
The given definition seems to have appeared for the first time in \cite[p. 180, footnote 4]{GS}. 
It also appears in many other places, e.g., in \cite[Sect. II.5.1, pp. 167-168]{Gr}. 

\begin{figure} [ht]
\begin{center}

\begin{tikzpicture}

    \tikzstyle{every node}=[draw, circle, fill=white, minimum size=4pt, inner sep=0pt, label distance=1mm]


    \draw (0,0)              node (x)  [label=right: $x$]          {}
        -- ++(90:1.5 cm)     node (ab) [label=above: $a \vee b$] {}
        -- ++(225:1.4142 cm) node (a)  [label=left: $a$]           {}
        -- ++(45:1.4142 cm) 
        -- ++(315:1.4142 cm) node (b)  [label=right: $b$]          {};
        
        \draw [dotted, very thick] (b)--(1,-1) node () [label=right: $b'$] {}
        -- ++ (x) 
        -- ++ (-1,-1) node () [label=left: $a'$] {}
        -- ++ (a);
 
\end{tikzpicture}

\end{center}
\caption{\label{Dmsl} Diagram for the usual notion of distributivity for join-semilattices}
\end{figure}
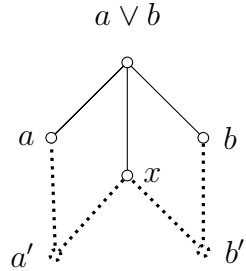

Next, note that {\bf(GS)} implies that every pair elements has a lower bound. 
In fact, we have the following equivalence. 

\begin{prop} \label{E1}
 Let ${\bf{J}} = (J; \leq)$ be a join-semilattice. 
 Then, the following two statements are equivalent: 
 
 (i) Every pair of elements has a lower bound. 
 
 (ii) for all $a, b, x \in J$, if $x \leq a \vee b$, 
 then there exist $a', b' \in J$ such that $a' \leq a$, $b' \leq b$, and $a' \vee b' \leq x$. 
\end{prop}

\begin{proof}
 (i) $\Rightarrow$ (ii) 
 Suppose $x \leq a \vee b$. 
 Let $a'$ be a lower bound of $\{ a, x\}$ and $b'$ be a lower bound of $\{ b, x\}$. 
 Then, $a' \leq a$ and $b' \leq b$. 
 Also, $a' \leq x$ and $b' \leq x$, which implies that $a' \vee b' \leq x$. 
 
 (ii) $\Rightarrow$ (i) 
 Let $a, b \in J$. 
 We have $a \leq a \vee b$. 
 Then, by hypothesis, there exist $a' \leq a$, $b' \leq b$ such that $a' \vee b' \leq a$. 
 As $b' \leq a' \vee b'$, it follows that $b' \leq a$. 
 Then, $b' \leq a, b$. 
 That is, $b'$ is a lower bound of $\{ a, b\}$. 
\end{proof}

This proposition shows that every GS-distributive join-semilattice is downward directed. 
This implies, as it is shown in \cite{Gr}, that the ideal $I \lor J$, generated by the union of two ideals $I, J$, 
is defined as in the case of distributive lattices, namely,
$$I \lor J = \{ a \lor b : a \in I, b \in J\}.$$

{
\noindent As a consequence, 
it follows that the ideals of a (GS)-distributive join-semilattice ${\bf{J}}$ form a lattice that will be denoted by $Id({\bf{J}})$, 
and Gr\"atzer proves in  \cite[p. 168]{Gr} the following characterization result. 

\begin{prop} \label{dJedi}
 Let ${\bf{J}}$ be a join-semilattice. 
 Then, ${\bf{J}}$ is (GS)-distributive iff $Id({\bf{J}})$ is distributive. 
\end{prop}
}

%

%
%
%
%
%

\subsection{K-distributivity}

The concept given in the following definition is similar to the one in {\bf(GS)}.

\begin{defi} \label{Kd}
A join-semilattice ${\bf{J}} = (J; \leq)$ is K-distributive iff 

\begin{itemize}
\item[{\bf (K)}] for all $a, b, x \in J$, if $x \leq a \vee b$, $x \nleq a$ and $x \nleq b$, 
then there exist $a', b' \in J$ such that $a' \leq a$, $b' \leq b$, and $x = a' \vee b'$. 
\end{itemize}
\end{defi}

\noindent In order to visualize, see again Figure \ref{Dmsl}. 
The given definition seems to have appeared for the first time in \cite[Definition 4, p. 122]{K}. 
It also appears, for example, in \cite[p. 167]{H}. 

It turns out that, from the very definition, GS-distributivity implies K-distributivity. 
In fact, as noted in \cite[1.5, p. 122-123]{K}, 
it is the case that GS-distributivity is equivalent to K-distributivity plus the condition that 
every pair of elements has a lower bound (that is, downward directed). 
Therefore, the following proposition makes clear the relationship between GS- and and K-distributivity. 

\begin{prop} GS-distributivity implies K-distributivity, but not conversely. 
\end{prop} 

The most simple counter-example showing that the reciprocal does not hold is the join-semilattice in Figure \ref{KGS}, that is not downward directed. 
Indeed, the given join-semilattice is K-distributive, 
as the only way to satisfy the antecedent of {\bf(K)} is to take $1 \leq a \vee b$, 
but then the consequent is also true.  
On the other hand, it is not (GS)-distributive, as we have $a \leq a \vee b$ and,  
however, there are no $a' \leq a$, $b' \leq b$ such that $a' \vee b' = a$.   

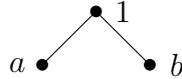
\begin{figure} [ht]
\begin{center}

\begin{tikzpicture}

    \tikzstyle{every node}=[draw, circle, fill=black, minimum size=4pt, inner sep=0pt, label distance=1mm]

    \draw (0,0)		node (a)	[label=left: $a$]	{}
        -- ++(45:1cm)	node (1)	[label=right: $1$]	{}
        -- ++(315:1cm)	node (b)	[label=right: $b$]	{};
 
\end{tikzpicture}

\end{center}
\caption{\label{KGS} Join-semilattice showing that K- does not imply GS-distributivity}
\end{figure}


{
Finally, analogously to Proposition \ref{dJedi}, we have the following characterisation of K-distributivity via ideals,  
a proof of which may be found in \cite[p. 123]{K}. 

 \begin{prop} \label{Kiff}
 Let ${\bf{J}}$ be a join-semilattice. 
 Then, ${\bf{J}}$ is (K)-distributive iff $Id({\bf{J}}) \cup \{\emptyset\}$ is distributive. 
\end{prop}
}

\subsection{H-distributivity} 

In \cite{H} Hickman introduces the concept of {\em mildly distributive} meet-semilattices 
as those meet-semilattices whose lattice of their strong ideals is distributive. 
In \cite[Theorem 2.5, p. 290]{H} it is stated that it is equivalent to the following statement:
\footnote{Note that the original Hickman's statement can be misleading since the condition 
``there exists $(x \wedge a_1) \vee (x \wedge a_2) \vee \cdots \vee (x \wedge a_n)$'' is missing.}

\smallskip

\begin{itemize}
\item[{\bf (H)}] for all $n$ and $x, a_1, \cdots, a_n$, 

IF for all $b$ (if $a_1 \leq b, \cdots, a_n \leq b$, then $x \leq b)$, 

THEN there exists $(x \wedge a_1) \vee \cdots \vee (x \wedge a_n)$ and $x \leq (x \wedge a_1) \vee \cdots \vee (x \wedge a_n)$.  
\end{itemize}

\smallskip

{
\noindent The given conditional may be seen as a translation of the following version of distributivity for lattices: 

\smallskip

\begin{center}
IF $x \leq a_1 \vee \cdots \vee a_n$, 
THEN $x \leq (x \wedge a_1) \vee \cdots \vee (x \wedge a_n)$. 
\end{center}
\smallskip

\noindent In the case of a join-semilattice ${\bf{J}} = (J; \leq)$ and using quantifiers, 
{\bf (H)} may be rendered as follows: 

\smallskip

\begin{itemize}
\item[] for all $n$ and $x, a_1, \dots, a_n \in J$,

IF $x \leq a_1 \vee \cdots \vee a_n$,  

THEN for all $y$, 
if 
for all $i=1, \ldots, n$ (for all $z$, IF $z \leq x$ and $z \leq a_i$, THEN $z \leq y$) 
%

\hspace*{2.9cm} then $x \leq y$
\end{itemize}
\smallskip

\noindent that is in turn equivalent to: 

\smallskip

\begin{itemize}
\item[] for all $n$ and $x, a_1, \dots, a_n \in J$,

IF $x \leq a_1 \vee \cdots \vee a_n$,  

THEN for all $y$, 
if 
(for all $z$, IF $z \leq x$ and ($z \leq a_1$ or \ldots or $z \leq a_n$), THEN $z \leq y$) 
%

\hspace*{2.9cm} then $x \leq y$. 
\end{itemize}

\smallskip

\noindent Using set-theoretic notation, {\bf (H)} may also be rendered as follows: 

\smallskip

\begin{itemize}
\item[{\bf (C)}] for all $n$ and $x, a_1, \dots, a_n \in J$,

if $x \leq a_1 \vee \cdots \vee a_n$, then $x \in (\{x, a_1 \}^l  \cup \cdots \cup \{x, a_n \}^l)^{ul}$. 
\end{itemize}

\smallskip

\noindent At this point, 
the reader may wonder whether the number $n$ of arguments is relevant or whether two arguments are enough. 
Let us settle this question. 
Firstly, with that in mind, consider  

\smallskip

\begin{itemize}
\item[{\bf (D$_{\vee_n}$)}] for all $x, a_1, \dots , a_n, c$, 

if $\{x,a_1\}^l  \cup \cdots \cup \{x,a_n\}^l \subseteq \{c \}^l$, 
then $\{x,a_1 \vee \cdots \vee a_n \}^l \subseteq \{c\}^l$. 
\end{itemize}

\noindent Now, let us state the following fact. 

\begin{lem}
 {\bf (D$_{\vee_n}$)} is equivalent to {\bf (C)}. 
\end{lem}

\begin{proof}
 $\Rightarrow$) Suppose $x \leq a_1 \vee \cdots \vee a_n$  
 and $y \in (\{x,a_1\}^l \cup \cdots \cup \{x,a_n\}^l)^u$. 
 Our goal is to see that $x \leq y$. 
 Take $c=y$ and apply  {\bf (D$_{\vee_n}$)}. 
 Then we have $\{x\}^l = \{x,a_1 \vee \cdots \vee a_n \}^l \subseteq \{y\}^l$, and hence $x \leq y$. 
%
%

 $\Leftarrow$) Suppose $\{x,a_1\}^l \cup \{x,a_2\}^l \cup \cdots \cup \{x,a_n\}^l \subseteq \{c \}^l$. 
 We have to prove that, if $y \leq x$ and $y \leq a_1 \vee \cdots \vee a_n$ then $y \leq c$. 
 Now, using {\bf (C)}, and the assumptions  $y \leq x$ and $y \leq a_1 \vee \cdots \vee a_n$ it follows that  
$y \in (\{x,a_1\}^l \cup  \cdots \cup \{x,a_n\}^l)^{ul}$. But  since $\{x,a_1\}^l \cup \{x,a_2\}^l \cup \cdots \cup \{x,a_n\}^l \subseteq \{c \}^l$, we also have 
$y \in (\{x,a_1\}^l \cup  \cdots \cup \{x,a_n\}^l)^{ul} \subseteq \{c\}^{lul} = \{c\}^l$. Hence $y \leq c$. 
%
%
\end{proof}

%

\noindent In turn, let us see that {\bf (D$_{\vee_n}$)} is equivalent to {\bf (D$_\vee$)}, 
which proves that having more than two arguments does not make any difference.  

\begin{lem}
 {\bf (D$_{\vee_n}$)} is equivalent to {\bf (D$_\vee$)}. 
\end{lem}

\begin{proof}
 We just prove that {\bf (D$_\vee$)} implies {\bf (D$_{\vee_3}$)}, the reciprocal being immediate. 
 Let us suppose $\{h,a_1\}^l \cup \{h, a_2\}^l \cup \{x, a_3\}^l \subseteq \{c \}^l$. 
 Then, we get both $\{h,a_1\}^l \subseteq \{c \}^l$ and $\{h, a_2\}^l \cup \{x,a_3\}^l \subseteq \{c \}^l$, 
 the last of which, using {\bf (D$_\vee$)}, implies that $\{h, a_2 \vee a_3 \}^l \subseteq \{c \}^l$, 
 which, together with the first, using {\bf (D$_\vee$)} again, 
 finally implies that $\{h, a_1 \vee a_2 \vee a_3 \}^l \subseteq \{c \}^l$. 
\end{proof}

\noindent As a consequence, 
H-distributivity coincides with the notion of DN-distributivity for join-semilattices introduced in Section \ref{DaND}. 
Accordingly, we have the following proposition.

\begin{prop} \label{Hd}
A join-semilattice ${\bf{J}} = (J; \leq)$ is H-distributive iff it is ND-distributive. 
\end{prop}
}

{
Analogously to Propositions \ref{dJedi} and \ref{Kiff}, 
we also have a characterization of H-distributivity for join-semilattices in terms of distributivity of a sublattice of their ideals. 
This appears as Corollary 2.4 in \cite[p. 290]{H}), 
where $Id_{fp}({\bf{J}})$ denotes the set  $\{(a_1] \cap \dots \cap (a_k] : a_1,...a_k \in J \}$, 
that is, the set of ideals that are intersection of a finite set of principal ideals of the join-semilattice ${\bf{J}}=(J; \leq)$. 

\begin{prop} \label{Hiff}
 Let ${\bf{J}}$ be a join-semilattice. 
 Then, ${\bf{J}}$ is H-distributive iff $Id_{fp}({\bf{J}})$ is distributive. 
\end{prop}
}

Let us now compare H- with K-distributivity. 

\begin{prop} \label{KH}
Let ${\bf{J}} = (J; \leq)$ be a join-semilattice. 
Then, K-distributivity implies H-distributivity.  
\end{prop}

\begin{proof}
Suppose 

\begin{itemize}
\item[(x1)] for all $x \in J$, if $x \leq h$ and $x \leq a$, then $x \leq c$ and

\item[(x2)] for all $x \in J$, if $x \leq h$ and $x \leq b$, then $x \leq c$.
\end{itemize}
\noindent Further, suppose both (S1) $x \leq h$ and (S2) $x \leq a \vee b$.
The goal is to prove $x \leq c$. 
Let us suppose that $x \leq a$. 
Then, using (x1) and (S1), it follows that $x \leq c$. 
The case $x \leq b$ is analogous using (x2). 
Finally, suppose both $x \nleq a$ and $x \nleq b$. 
Using (K) and (S2) it follows that there exist $a', b' \in J$ such that $a' \leq a$, $b' \leq b$, and (F) $x = a' \vee b'$, 
which implies $a' \leq x$, which using (S1) gives $a' \leq h$. 
As we also have $a' \leq a$, using (x1) we get $a' \leq c$. 
Reasoning analogously, we get $b' \leq c$. 
So, using (F) it follows that $x \leq c$.  
\end{proof}

The reciprocal of Proposition \ref{KH} does not hold considering the model in Figure \ref{HK} 
(with the understanding that there is no element in the white node). 
The given model appears as a poset in \cite[Figure 2.7, p. 37]{Go}.\footnote{ 
We thank the author of that paper for communicating this example.}  
We provide a proof using the characterization of K- and H-distributivity by their ideals (Propositions \ref{Kiff} and  \ref{Hiff}).

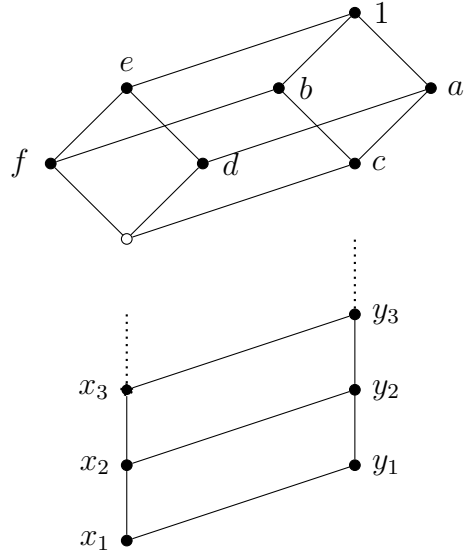
\begin{figure} [ht]
\begin{center}

\begin{tikzpicture}

    \tikzstyle{every node}=[draw, circle, fill=black, minimum size=4pt, inner sep=0pt, label distance=1mm]
        
        \draw			(0,0)		node (x1)	[label=left: $x_1$]	{};
        \draw			(x1)--(3,1)	node (y1)	[label=right: $y_1$]	{};
        \draw			(y1)--(3,2)	node (y2)	[label=right: $y_2$]	{};
        \draw			(y2)--(3,3)	node (y3)	[label=right: $y_3$]	{};
        \draw [dotted, thick]	(y3)--(3,4); 	
        \draw 			(3, 5) 		node (c) 	[label=right: $c$]	{};
        \draw			(c)--(4,6)	node (a)	[label=right: $a$]	{};
        \draw			(a)--(3,7)	node (top)	[label=right: $1$]	{};
        
        \draw 			(top)--(0,6)	node (e)	[label=above: $e$]	{};
        \draw 			(e)--(1,5)	node (d)	[label=right: $d$]	{};
        \draw 			(d)--(a);
        \draw 			(c)--(2,6)	node (b)	[label=right: $b$]	{};
        \draw 			(b)--(top);
        \draw 			(e)--(-1,5)	node (f)	[label=left: $f$]	{};
        \draw 			(b)--(f);
        \draw 			(f)--(0,4)	node (df)	[fill=white]		{};
        \draw 			(c)--(df);
        \draw 			(d)--(df);
        
        \draw [dotted, thick]	(0,3)--(0,2)	node (x3) 	[label=left: $x_3$] 	{};
        \draw			(x3)--(0,1)	node (x2)	[label=left: $x_2$]	{};
        \draw			(x2)--(x1);
        \draw			(x2)--(y2);
        \draw			(x3)--(y3);
 
\end{tikzpicture}

\end{center}
\caption{\label{HK} H-distributive, but not K-distributive join-semilattice}
\end{figure}

{
\begin{prop}
 H-distributivity does not imply K-distributivity. 
\end{prop}

\begin{proof} 
 Let us characterize the sets $Id_{fp}(J)$ and $Id(J)$, where $(J, \leq)$ is the join-semilattice of Figure \ref{HK}. 
 An easy computation proves, on the one hand, that $Id_{fp}(J)$ is isomorphic to the ordered set of Figure \ref{HK} plus the ideal $I_{\overline{x}} = (f] \wedge (d]$,  
 whose elements are $\{x_i : i \in w\}$, that does not exist in the original join-semilattice. 
 On the other hand, $Id(J)$ is the set of ideals in $Id_{fp}(J)$ plus the ideal ${I}_{\overline{y}}$ generated by the set $\{y_i : i \in w\}$, 
 that is, the ideal with elements ${I}_{\overline{y}} = \{y_i : i \in w\} \cup \{x_i : i \in w\}$. 
 Clearly, this ideal is not a finite intersection of principal ideals. 
 Both $Id_{fp}(J)$ and $Id(J)$ are lattices. 
 Moreover, it is obvious that $Id_{fp}(J)$ is a distributive lattice and thus the join-semilattice of the example is H-distributive. 
 But this is not the case for $Id(J)$, 
 since it has a sublattice isomorphic to the pentagon formed by the elements $(a]$, $(d]$, $(c]$, ${I}_{\overline{y}}$, and ${I}_{\overline{x}}$. 
 Thus, the join-semilattice of the example is not K-distributive.
\end{proof}

}

%
%
%
%
%
%

\smallskip

It is natural to ask whether it is possible to find a finite example in order to prove the reciprocal of Proposition \ref{KH}. 
Let us see that the answer is negative. 

\begin{prop}
For finite join-semilattices,  H-distributivity and K-distributivity coincide. 
\end{prop}

\begin{proof}
Consider a finite H-distributive join-semilattice. 
We want to see that it is K-distributive. 
Accordingly, suppose $x \leq a \vee b$, $x \nleq a$, and $x \nleq b$. 
It is natural to consider $\bigvee \{ a, x\}^l$ and $\bigvee \{ b, x\}^l$ 
as candidates for $a'$ and $b'$ in the definition of K-distributivity. 
Now, in order to do that, we first need to prove that the sets $\{ a, x\}^l$ and $\{ b, x\}^l$ are not empty. 
Suppose, say, $\{ a, x\}^l = \emptyset$. 
Then, we have :

\begin{itemize}
\item[] - for all $y$, if $y \leq x$ and $y \leq a$, then $y \leq b$ (as $\{ a, x\}^l = \emptyset$), 

-  for all $y$, if $y \leq x$ and $y \leq b$, then $y \leq b$,  

-  $x \leq x$, and 

-  $x \leq a \vee b$. 
\end{itemize}
\noindent So, using H-distributivity, it follows that $x \leq b$, a contradiction.

Having proved that both $\{ a, x\}^l \neq \emptyset$ and $\{ b, x\}^l \neq \emptyset$, 
let us note that both $\bigvee \{ a, x\}^l$ and $\bigvee \{ b, x\}^l$ exist, due to having a finite structure. 
Next, let us see that $\bigvee \{ a, x\}^l =$ inf $\{ a, x \}$ (analogously, $\bigvee \{ b, x\}^l =$ inf $\{ b, x \}$). 
It is clear that both $\bigvee \{ a, x\}^l \leq a$ and $\bigvee \{ a, x\}^l \leq x$. 
Now, suppose $y \leq a, x$. 
Then, $y \in \{ a, x\}^l$, and so, $y \leq \bigvee \{ a, x\}^l$, as desired.   

It remains to be seen that 1) inf $\{ a, x \} \leq a$, 2) inf $\{ b, x \} \leq b$, 
and 3) $x =$ inf$\{ a, x \} \vee$ inf$\{ b, x \}$. 
Now, 1) and 2) are easy to see. 
Regarding 3), as we have both that inf$\{ a, x \} \leq x$ and sup $\{ b, x \} \leq x$, 
it follows that inf$\{ a, x \} \ \vee$ \ inf$\{ b, x \} \leq x$. 
Finally, observe that the inequality $x \leq$ inf$\{ a, x \} \ \vee$ inf$\{ b, x \}$ 
follows from 

\begin{itemize}
\item[] - for all $y$, if $y \leq x$ and $y \leq a$, then $y \leq$ inf$\{ a, x \} \ \vee$ inf$\{ b, x \}$, 

- for all $y$, if $y \leq x$ and $y \leq b$, then $y \leq$ inf$\{ a, x \} \ \vee$ inf$\{ b, x \}$,  

- $x \leq x$, and 

- $x \leq a \vee b$, 
\end{itemize}
using H-distributivity.
\end{proof}

In fact, it is easy to observe that in the case of a finite join-semilattice $J$, 
the sets of ideals $Id(J)$ and $Id_{fp}(J)$ coincide since, 
for any two elements $a, b$, either there is no lower bound, 
that is, $\{a, b\}^l = \emptyset$, or there exists their meet $a \land b = \bigvee \{ a, b\}^l$.

\subsection{LR-distributivity}

Larmerov\'a-Rach\r{u}nek version of distributivity (see \cite{LR}) was given for posets, as we next see.  

\begin{defi} \label{LRPd}
A poset ${\bf{P}} = (P; \leq)$ is \emph{LR-distributive} iff 

\begin{itemize}
\item[{\bf (LRP)}]  for all $a, b, c \in P$, 
$(\{ c, a \}^l \cup \{ c, b \}^l)^{ul} = (\{ c \} \cup \{ a, b \}^u)^l$. 
\end{itemize}
\end{defi}

\begin{rem}
 In the given definition, it is enough to take one inclusion. 
 Indeed, given a poset ${\bf{P}} = (P; \leq)$ and $a, b, c \in P$, it is always the case that 
 $(\{ c, a \}^l \cup \{ c, b \}^l)^{ul} \subseteq (\{ c \} \cup \{ a, b \}^u)^l$. 
\end{rem}

It is natural to ask for LR-distributivity in the case of a join-semilattice. 
The following definition follows from the fact that 
in a join-semilattice ${\bf{J}} = (J, \leq)$ it holds that $(\{ c \} \cup \{ a, b \}^u)^l = \{ c,  a \vee b \}^l$. 

\begin{defi} \label{LRJd}
A join-semilattice ${\bf{J}} = (J, \leq)$ is \emph{LR-distributive} iff 

\begin{itemize}
\item[{\bf (LR)}] for all $a, b, c \in J$, 
$\{ c,  a \vee b \}^l \subseteq (\{ c, a \}^l \cup \{ c, b \}^l)^{ul}$. 
\end{itemize}
\end{defi}

Now, it can be seen that LR-distributivity is equivalent to H-distributivity, and hence to the condition {\bf{(D$_{\vee}$)}} as well. 

\begin{prop} 
 Let ${\bf{J}} = (J; \leq)$ be a join-semilattice. Then the following conditions are equivalent: 
 \begin{itemize}
 \item[(i)] ${\bf{J}}$ satisfies {\bf(LR)},
 \item[(ii)]  ${\bf{J}}$ satisfies {\bf(H)},
 \item[(iii)] ${\bf{J}}$ satifies  {\bf{(D$_{\vee}$)}.}
 \end{itemize}
\end{prop}

\begin{proof}
The equivalence between (ii) and (iii) is Prop. \ref{Hd}. Let us prove that {\bf(LR)} implies {\bf(H)}. 
 Suppose 
 \begin{itemize}
\item[] (x1)  for all $x \in J$, if $x \leq h$ and $x \leq a$, then $x \leq c$,

 (x2) for all $x \in J$, if $x \leq h$ and $x \leq b$, then $x \leq c$,  $x \leq h$ and $x \leq a \vee b$. 
 \end{itemize}
 Then, the last two inequalities imply $x \in \{ c,  a \vee b \}^l$. 
 So, using {\bf(LR)} we get that $x \in (\{ c, a \}^l \cup \{ c, b \}^l)^{ul}$.
 That is, for all $y \in J$, if $y \in (\{ c, a \}^l \cup \{ c, b \}^l)^u$, then $x \leq y$. 
 Now, it should be clear that (x1) and (x2) imply that $c \in (\{ c, a \}^l \cup \{ c, b \}^l)^u$. 
 So, $x \leq c$, as desired. 
 
 Now, let us see that {\bf(H)} implies {\bf(LR)}. 
 Suppose $x \in \{ h,  a \vee b \}^l$, that is, (H1) $x \leq h$ and (H2) $x \leq a \vee b$. 
 In order to get our goal, that is, $x \in (\{ h, a \}^l \cup \{ h, b \}^l)^{ul}$, 
 let us suppose that (S) $y \in (\{ h, a \}^l \cup \{ h, b \}^l)^u$ and try to derive $x \leq y$. 
 Now, (S) means that for all $z \in J$, if $z \in (\{ h, a \}^l \cup \{ h, b \}^l$, then $z \in y$, that is,
 
  \begin{itemize}
\item[]  (y1) for all $z \in J$, if $z \leq h$ and $z \leq a$, then $z \leq y$ and
 
 (y2) for all $z \in J$, if $z \leq h$ and $z \leq b$, then $z \leq y$. 
 \end{itemize}
 
\noindent Now, using {\bf(H)}, (y1), (y2), (H1), and (H2), we get our goal, that is, $x \leq y$.  
\end{proof}

\subsection{B-distributivity}

The following definition seems to have appeared for the first time in \cite[Theorem 2.2. (i), p. 261]{B}. 

\begin{defi} \label{Bd}
A join-semilattice ${\bf J} = (J; \leq)$ is \emph{B-distributive} iff 

\begin{itemize}
\item[{\bf (B)}] for all $n, a_1, a_2, \dots, a_n, x \in J$, 
if $a_1 \wedge a_2 \wedge \cdots \wedge a_n$ exists, 
then also $(x \vee a_1) \wedge (x \vee a_2) \wedge \cdots \wedge (x \vee a_n)$ exists 
and equals $x \vee (a_1 \wedge a_2 \wedge \cdots \wedge a_n)$. 
\end{itemize}
\end{defi}

We have the following fact. 

\begin{prop} \label{HB} 
 Let ${\bf{J}} = (J; \leq)$ be a join-semilattice. 
 Then, H-distributivity implies B-distributivity.  
\end{prop}

\begin{proof}
 Let us have a H-distributive join semilattice $J$ and let us take $a, b, x \in J$ 
 (the general case follows by induction). 
 Let us suppose that $a \wedge b$ exists in $J$. 
 Then, also $x \vee (a \wedge b)$ exists in $J$. 
 Our goal is to see that $x \vee (a \wedge b) =$ inf $\{x \vee a, x \vee b \}$. 
 It is clear that $x \vee (a \wedge b) \leq x \vee a, x \vee b$. 
 Now, suppose both (F1) $y \leq x \vee b$ and (F2) $y \leq x \vee a$. 
 We have to see that $y \leq x \vee (a \wedge b)$. 
 It immediately follows that  
 
 \begin{itemize}
\item[]  (x1) for all $w \in J$, if $w \leq x \vee b$ and $w \leq x$, then $w \leq x \vee (a \wedge b)$.
 \end{itemize}
\noindent  Now, suppose (F3) $w \leq x \vee b$ and (F4) $w \leq a$. 
 Then, we have both 
 
 \begin{itemize}
\item[]  (x1') for all $y \in J$, if $y \leq a$ and $y \leq x$, then $y \leq x \vee (a \wedge b)$, and 
 
 (x2') for all $y \in J$, if $y \leq a$ and $y \leq b$, then $y \leq x \vee (a \wedge b)$.
 \end{itemize}
 \noindent So, applying H-distributivity to (F3), (F4), (z1'), and (x2'), we have $w \leq x \vee (a \wedge b)$. 
 That is, we have proved 
 
 \begin{itemize}
\item[]  (x2) for all $w \in J$, if $w \leq x \vee b$ and $w \leq a$, then $w \leq x \vee (a \wedge b)$. 
  \end{itemize}
 \noindent Using H-distributivity, (F1), (F2), (x1) and (x2), 
 it finally follows that $y \leq x \vee (a \wedge b)$, as desired. 
\end{proof}

The reciprocal of Proposition \ref{HB} does not hold as may be seen in Figure \ref{BH}. 

\begin{figure} [ht]
\begin{center}

\begin{tikzpicture}

    \tikzstyle{every node}=[draw, circle, fill=black, minimum size=4pt, inner sep=0pt, label distance=1mm]

    \draw (0,0)			node (a)[label=left: $a$]	{}
        -- ++(45:1.4142cm)	node (1)[label=right:$1$]	{}
        -- ++(270:1cm)		node (b)[label=right:$b$]	{}
        -- ++(90:1cm)
        -- ++(315:1.4142cm)	node (d)[label=right: $c$]	{};
            
\end{tikzpicture}

\end{center}
\caption{\label{BH} Join-semilattice showing that B- does not imply H-distributivity}
\end{figure}
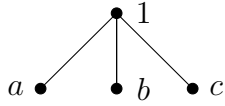

Observe also that the lattice $Id_{fp}(J)$, for $J$ being the join-semilattice of  Figure \ref{BH}, is not distributive since it is a diamond.

\subsection{S$_n$-distributivity}

The following definition seems to have appeared for the first time in \cite{S}. 

\begin{defi} \label{Sd}
A join-semilattice $(J; \leq)$ is said to be \emph{S$_n$-distributive} for $n$ a natural number, $2 \leq n$, iff 

\begin{itemize}
\item[{\bf (S$_n$)}] for all $a_1, a_2, \dots, a_n, x \in J$, 
if $a_1 \wedge a_2 \wedge \cdots \wedge a_n$ exists, 
then also $(x \vee a_1) \wedge (x \vee a_2) \wedge \cdots \wedge (x \vee a_n)$ exists 
and equals $x \vee (a_1 \wedge a_2 \wedge \cdots \wedge a_n)$. 
\end{itemize}
\end{defi}

It is easy to see that B-distributivity implies $S_n$-distributivity, for any $n \geq 2$. 
It is also clear that for any $n \geq 2$, $S_{n+1}$ implies $S_n$. 
On the other hand, we have that for no natural $n \geq 2$ it holds that $S_n$-distributivity implies B-distributivity. 
In fact, it was proved that for any $n \geq 2$, $S_n$ does not imply $S_{n+1}$ (see \cite{Ke}), 
where infinite models using the real numbers are provided. 
As in the case of GS- and H-distributivity, it is natural to ask whether, for example, finite models are possible. 
As in the cases just mentioned, the answer is negative as already proved in \cite[Theorem 7.1, p. 1071]{SCLS}.   
In \cite[Theorem, p. 26]{SA} it is also proved that it is not possible to find infinite wellfounded models. 

Therefore, so far we have seen that, in the case of a join-semilattice, we have the following chain of implications: 

\smallskip

\begin{center}
{\bf (GS)} $\Rightarrow$ {\bf (K)} $\Rightarrow$ {\bf (H)} $\Leftrightarrow$ {\bf (LR)} $\Leftrightarrow$ {\bf (ND)} $\Rightarrow$ {\bf (B)} $\Rightarrow \cdots$ {\bf (S$_n$)} 
$\Rightarrow$ {\bf (S$_{n-1}$)} $\Rightarrow \cdots$ {\bf (S$_2$)}.  
\end{center}

%

\section{Join-semilattices with arrow}

The expansion of semilattices with an arrow operation has been well studied in the literature in the case of meet-semilattices 
under the name of relatively pseudo-complemented semilattices (see, for example, \cite{Gr}). 
However, as far as we know, the expansion of join-semilattices with an arrow has not received much attention, see, for instance, \cite{Chajda, Chajda2}. 
In this section we deal with distributivity of join-semilattices expanded with an arrow operation. 

A join-semilattice with arrow is a structure $(J; \leq, \to)$ where $(J; \leq)$ is a join-semilattice and 
the arrow $\to$ is a binary operation such that for all $a, b \in J$: 

\smallskip

$a \to b = \max \{c \in J:$ for all $x \in J$, if $x \leq a$ and $x \leq c$, then $x \leq b \}$.  

\smallskip

\noindent The existence of the $\to$ operation is clearly equivalent to the requirement that $\to$ satisfies the following two conditions: 

\smallskip

($\to$E) for all $x \in J$, if $x \leq a$ and $x \leq a \to b $, then $x \leq b$, 

\smallskip

($\to$I) for all $c \in J$, IF for all $x \in J$, if $x \leq a$ and $x \leq c$, then $x \leq b$, THEN $c \leq a \to b$.  

\begin{rem}
 The idea of defining arrow in a poset was already present in \cite{Ha} (see Definition 4, 
 where the author uses the terminology of Brouwer poset and also proves that a poset with arrow is LR-distributive). 
 Moreover, the author, using LR-notation, defines $a \to b =$ max $\{c \in J: \{a,c\}^l \subseteq \{b\}^l\}$. 
\end{rem}

\begin{rem}
 In a lattice, or even in a meet-semilattice, arrow coincides with the usual relative meet-complement. 
 This follows from the fact that, as previously mentioned, 
 the inequality $a \wedge x \leq b$ is equivalent to the following universal quantification: 
 for all $y$, if $y \leq a$ and $y \leq x$, then $y \leq b$. 
 By the way, we prefer to use ``arrow'' instead of ``relative meet-complement'', because the meet is not present. 
\end{rem}

As is well known, a lattice with a relative meet-complement (that is in fact a Heyting algebra) is distributive (see \cite{S1} or \cite{S2}). 
The natural question arises whether a join-semilattice with arrow is distributive 
in any of the senses considered in Section \ref{SDN}. 
The answer is negative in the case of (GS)-distributivity, 
as the join-semilattice in Figure \ref{jsl1} has arrow and is not GS-distributive.

\begin{figure} [ht]
\begin{center}

\begin{tikzpicture}

    \tikzstyle{every node}=[draw, circle, fill=white, minimum size=4pt, inner sep=0pt, label distance=1mm]
     
     \draw (0,0) node (a)         [label=left: $a$] {}

        -- ++(45:1cm)  node (1) [label=above: $1$] {}
        -- ++(315:1cm) node (b) [label=right:$b$] {};
   
\end{tikzpicture} \hskip 2cm
\begin{tabular}{c||c|c|c}

 $\to$&a&b&1 \\
 \hline \hline
 a&1&b&1 \\
 \hline 
 b&a&1&1 \\
 \hline 
 1&a&b&1 \\
\end{tabular}

\end{center}
\caption{\label{jsl1} A join-semilattice with arrow}
\end{figure}
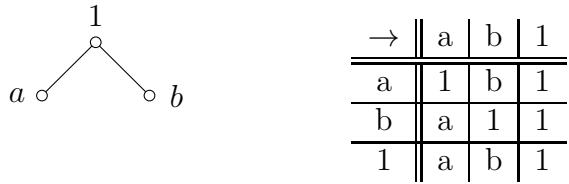

A similar question in the case of K-distributivity has also a negative answer, 
as the the join-semilattice in Figure \ref{jsl2}, already given in Figure \ref{HK}, 
has arrow and is not K-distributive. 

\begin{figure} [ht]
\begin{center}

\begin{tikzpicture}

    \tikzstyle{every node}=[draw, circle, fill=black, minimum size=4pt, inner sep=0pt, label distance=1mm]
        
        \draw			(0,0)		node (x1)	[label=left: $x_1$]	{};
        \draw			(x1)--(3,1)	node (y1)	[label=right: $y_1$]	{};
        \draw			(y1)--(3,2)	node (y2)	[label=right: $y_2$]	{};
        \draw			(y2)--(3,3)	node (y3)	[label=right: $y_3$]	{};
        \draw [dotted, thick]	(y3)--(3,4); 	
        \draw 			(3, 5) 		node (c) 	[label=right: $c$]	{};
        \draw			(c)--(4,6)	node (a)	[label=right: $a$]	{};
        \draw			(a)--(3,7)	node (top)	[label=right: $1$]	{};
        
        \draw 			(top)--(0,6)	node (e)	[label=above: $e$]	{};
        \draw 			(e)--(1,5)	node (d)	[label=right: $d$]	{};
        \draw 			(d)--(a);
        \draw 			(c)--(2,6)	node (b)	[label=right: $b$]	{};
        \draw 			(b)--(top);
        \draw 			(e)--(-1,5)	node (f)	[label=left: $f$]	{};
        \draw 			(b)--(f);
        \draw 			(f)--(0,4)	node (df)	[fill=white]		{};
        \draw 			(c)--(df);
        \draw 			(d)--(df);
        
        \draw [dotted, thick]	(0,3)--(0,2)	node (x3) 	[label=left: $x_3$] 	{};
        \draw			(x3)--(0,1)	node (x2)	[label=left: $x_2$]	{};
        \draw			(x2)--(x1);
        \draw			(x2)--(y2);
        \draw			(x3)--(y3);
 
\end{tikzpicture}

\begin{tabular}{c||c|c|c|c|c|c|c|c|c|c|c|c|c}

 $\to$&$x_1$&$x_2$&$x_n$&$y_1$&$y_2$&$y_n$&f&d&e&c&b&a&1 \\
 \hline \hline
 $x_1$&1&1&1&1&1&1&1&1&1&1&1&1&1 \\
 \hline 
 $x_2$&$y_1$&1&1&$y_1$&1&1&1&1&1&1&1&1&1 \\
 \hline 
 $x_n$&$y_1$&$y_2$&1&$y_1$&$y_2$&1&1&1&1&1&1&1&1 \\
 \hline 
 $y_1$&e&e&e&1&1&1&e&e&e&1&1&1&1 \\
 \hline 
 $y_2$&e&e&e&$y_1$&1&1&e&e&e&1&1&1&1 \\
 \hline 
 $y_n$&$x_1$&$x_2$&e&$y_1$&$y_2$&1&1&1&1&1&1&1&1 \\
 \hline 
 f&$y_1$&$y_2$&$y_n$&$y_1$&$y_2$&$y_n$&1&a&1&c&1&a&1 \\
 \hline 
 d&$y_1$&$y_2$&$y_n$&$y_1$&$y_2$&$y_n$&b&1&1&b&b&1&1 \\
 \hline 
 e&$y_1$&$y_2$&$y_n$&$y_1$&$y_2$&$y_n$&b&a&1&c&b&a&1 \\
 \hline 
 c&$x_1$&$x_2$&$x_n$&$y_1$&$y_2$&$y_n$&e&e&e&1&1&1&1 \\
 \hline 
 b&$x_1$&$x_2$&$x_n$&$y_1$&$y_2$&$y_n$&e&d&e&a&1&a&1 \\
 \hline 
 a&$x_1$&$x_2$&$x_n$&$y_1$&$y_2$&$y_n$&f&e&e&b&b&1&1 \\
 \hline 
 1&$x_1$&$x_2$&$x_n$&$y_1$&$y_2$&$y_n$&f&d&e&c&b&a&1 \\
\end{tabular}

\end{center}
\caption{\label{jsl2} Another join-semilattice with arrow}
\end{figure}
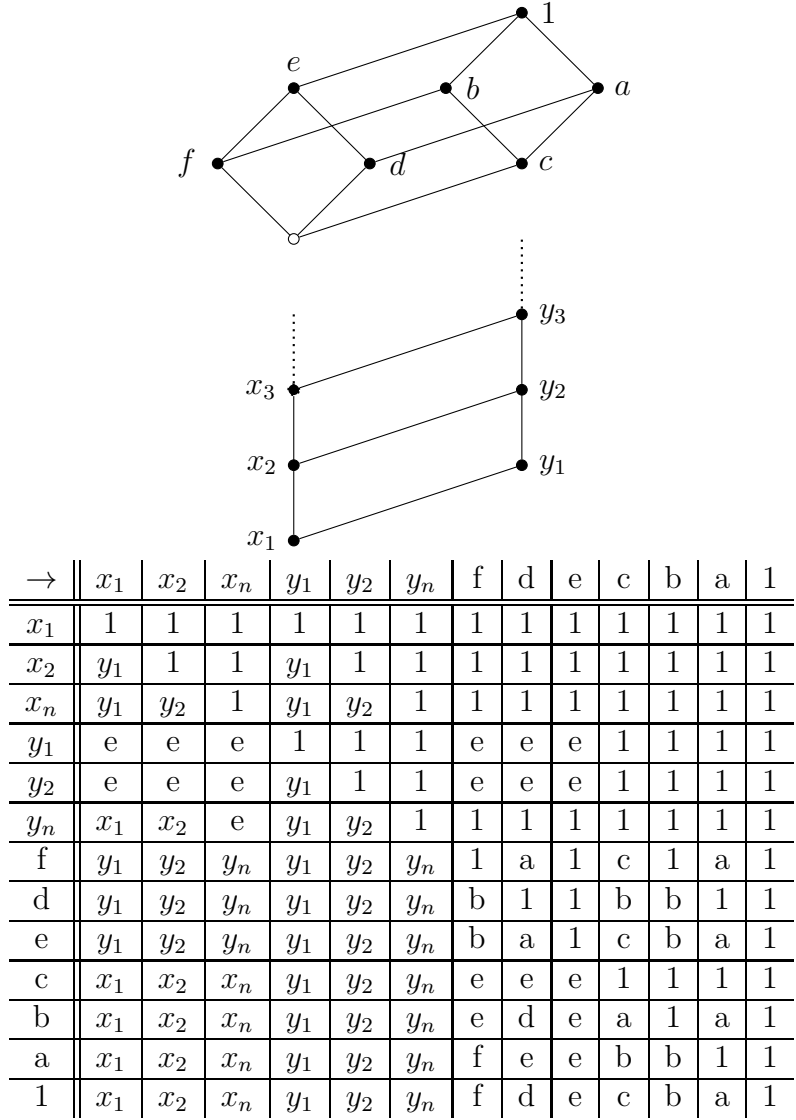

The case of H-distributivity is different, as we see next.  

\begin{prop}
 Every join-semilattice expanded with arrow is H-distributive.
\end{prop}

\begin{proof}
 Let ${\bf J} = (J; \leq)$ be a join-semilattice with arrow.  
 Take $a, b, c, h \in J$. 
 Suppose
 \begin{itemize}
\item[] (x1) for all $x \in J$, if $x \leq h$ and $h \leq a$, then $h \leq c$ and 
 
 (x2) for all $x \in J$, if $x \leq h$ and $x \leq b$, then $x \leq c$. 
  \end{itemize}

\noindent Take $y \in J$ and suppose 
  \begin{itemize}
\item[]
 (F1) $y \leq h$ and 
 
 (F2) $y \leq a \vee b$.  
 \end{itemize}
\noindent Now, using ($\to$I), (x1) implies $a \leq h \to c$ and (x2) implies $b \leq h \to c$.  
These inequalities together with (F2) imply $y \leq h \to c$, 
which, using (F1) and ($\to$E), gives $y \leq c$. 
\end{proof}

Analogously to what happens when considering lattices, in the finite case we have the following fact.

\begin{prop} \label{fHd}
 Every finite H-distributive join-semilattice has arrow.
\end{prop}

\begin{proof}
 Let ${\bf J} = (J; \leq)$ be a finite H-distributive join-semilattice.  
 Due to finiteness, 
 $c_1 \vee c_2 \vee \cdots \vee c_n = \bigvee \{c \in J:$ 
 for all $x \in J$, if $x \leq a$ and $x \leq c$, then $x \leq b \}$ exists, for any $a, b\in J$. 
 It is clear that for any $c_i, 1 \leq i \leq n$, it holds that 
  \begin{itemize}
\item[]
 (F) for all $x$, if $x \leq a$ and $x \leq c_i$, then $x \leq b$.
 \end{itemize}
\noindent Now, let us see that $c_1 \vee c_2 \vee \cdots \vee c_n$ is in fact $a \to b$. 

First, let us see that $c_1 \vee c_2 \vee \cdots \vee c_n \in \{c \in J:$ 
for all $x \in J$, if $x \leq a$ and $x \leq c$, then $x \leq b \}$. 
That is, we have to see that 
 \begin{itemize}
\item[]
 (T) for all $x \in J$, if $x \leq a$ and $x \leq c_1 \vee c_2 \vee c_n$, then $x \leq b$.
\end{itemize}
\noindent Now, (T) clearly follows from (F) by H-distributivity. 

Secondly, let us take $c \in J$ such that for all $x \in J$, if $x \leq a$ and $x \leq c$, then $x \leq b$. 
Then, obviously, $c \in \{c \in J:$ for all $x \in J$, if $x \leq a$ and $x \leq c$, then $x \leq b \}$. 
Then, $c \leq c_1 \vee c_2 \cdots \vee c_n$, 
as $c_1 \vee c_2 \vee \cdots \vee c_n = \bigvee \{c \in J:$ for all $x \in J$, if $x \leq a$ and $x \leq c$, then $x \leq b \}$.  
\end{proof}

%
%
%
%
%
%
%

Finally, the natural question arises whether the class of join-semilattices expanded with arrow 
forms a variety or at least a quasi variety. 
The following example proves that the answer is negative. 
Indeed, consider the distributive lattice in Figure \ref{B9}, 
which is the direct product ${\bf J} = (L \times L; \leq)$ where  $L = \{0, \tfrac{1}{2}, 1\}$. 
It is clear that we can define in ${\bf J}$ an arrow $\to$, in fact, ${\bf J}^* = (L \times L; \leq, \to)$ becomes a Heyting algebra. 
Now, consider ${\bf J}^*$ as a join-semilattice with arrow, 
and observe that the set $B$ of elements represented by black nodes in the figure is the domain of a subalgebra $(B; \leq, \to)$ of  ${\bf J}^*$, 
since both $\lor$ and $\to$ are closed on $B$. 
However, the join-semilattice $(B, \leq)$ is not distributive (it contains a pentagon), 
and moreover the arrow operation is not defined for all pairs of elements. 
In particular, $(\tfrac{1}{2}, \tfrac{1}{2}) \Rightarrow (0, 0)$ is not defined since the set 
$$\{ (c, d) \in B : \forall (x,y) \in B, \mbox{if } (x, y) \leq (c, d) \mbox{ and }(x, y)  \leq (\tfrac{1}{2}, \tfrac{1}{2}) \mbox{, then } (x, y) \leq (0, 0) \}$$ 
has no maximum.

%
%
%
%
%
%

\begin{figure} [ht]
\begin{center}

\begin{tikzpicture}

    \tikzstyle{every node}=[draw, circle, fill=black, minimum size=4pt, inner sep=0pt, label distance=0.5mm]

    \draw (0,0)	node (00) [label=right:{(0,0)}] {};
    \draw (-1,1)node (0m) [label=left:{(0,$\tfrac{1}{2}$)}] [fill=white] {};
    \draw (-2,2)node (01) [label=left:{(0,1)}] []{};
    \draw (1,1) node (m0) [label=right:{($\tfrac{1}{2}$,0)}] [fill=white] {};
    \draw (2,2)	node (10) [label=right:{(1,0)}] {};
    \draw (0,2)	node (mm) [label=right:{($\tfrac{1}{2}$,$\tfrac{1}{2}$)}] {};
    \draw (-1,3)node (m1) [label=left:{($\tfrac{1}{2}$,1)}]{};
    \draw (1,3)	node (1m) [label=right:{(1,$\tfrac{1}{2}$})]{};
    \draw (0,4)	node (11) [label=right:{(1,1)}]{};
    
    \draw (00)--(0m)--(01)--(11)--(1m)--(10)--(m0)--(00)
    	  (0m)--(mm)--(1m) 
    	  (m0)--(mm)--(m1);
 
\end{tikzpicture}

\end{center}
\caption{\label{B9} A distributive join-semilattice with definable arrow. }
\end{figure}
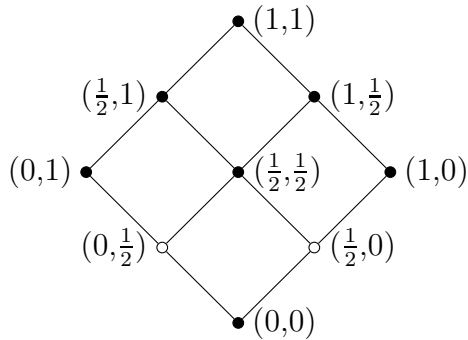




\section{Conclusions}

In this paper we have proposed a notion of distributivity for join-semilattices with logical motivations 
related to Gentzen's disjunction elimination rule in the $\{\lor, \to\}$-fragment of intuitionistic logic, 
and we have compared it to other notions of distributivity for join-semilattices proposed in the literature. 

There are a number of open problems that we plan to address as future research. 
In particular we can mention the following ones: 

\begin{itemize}

\item As for the logical motivation, similar to the $\bf (\lor E)$ rule in Section 3, 
one can consider the following rule with two contexts: 

\begin{prooftree}
 \AxiomC{$\mathfrak H_1$, $\mathfrak H_2$, $\mathfrak{A \vee B}$}
 \AxiomC{$\mathfrak H_1$, $\mathfrak H_2$, [$\mathfrak A$]}
 \noLine
 \UnaryInfC{$\mathfrak C$}
 \AxiomC{$\mathfrak H_1$, $\mathfrak H_2$, [$\mathfrak B$]}
 \noLine
 \UnaryInfC{$\mathfrak C$}
 \TrinaryInfC{$\mathfrak C$}
\end{prooftree}

\noindent This rule also has a natural algebraic translation in the case of join-semilattices. 
The question arises whether it is equivalent to the condition {\bf{(D$_{\vee}$)}} or if it leads to a different one. 

\item Distributive lattices are charecterized by their lattice of ideals. 
In the case of join-semillatices, there are similar characterizations for GS-, K- and H-distributivity, 
but not for B- and S$_n$distributivity. 
The question is whether B- and S$_n$-distributive join-semilattices can be characterized by means of their ideals.

\item In \cite{Chajda3} the authors generalize the well-known characterisation of distributive lattices in terms of forbidden sublattices 
(diamond and pentagon) to distributive posets, also identifying the set of forbidden subposets. 
A similar study for distributive join-semilattices is an open question.

\end{itemize}

\subsection*{Acknowledgments} The authors acknowledge partial support by the H2020 MSCA-RISE-2015 project SYSMICS. Esteva and Godo also acknowledge the FEDER/MINECO project TIN2015-71799-C2-1-P.

\end{document}